\documentclass[12pt]{amsart}

\usepackage{amsmath}
\usepackage{amssymb}
\usepackage{amsthm}
\usepackage[margin=1in, top=1.5in]{geometry}

\usepackage[slantedGreek]{mathptmx}
\usepackage{url}
\usepackage{bbm}

\newcommand{\id}{\mathbbm{1}}
\newcommand{\supp}{\mbox{\rm supp}\; }
\newcommand{\grad}{\nabla{}}
\renewcommand{\div}{\mbox{\rm div}}

\newtheorem{theorem}{Theorem}
\newtheorem{lemma}{Lemma}
\newtheorem{remark}{Remark}

\newtheorem{definition}{Definition}
\newtheorem*{theorem*}{Theorem}
\newtheorem{hypothesis}{Hypothesis}

\newcommand{\e}{\varepsilon}

\renewcommand{\span}{\mbox{\rm span}}
\newcommand{\barl}{\bar{l}}
\newcommand{\proj}{\mbox{\rm proj}}
\newcommand{\esssup}{\mbox{\rm ess\,sup}}
\newcommand{\E}{E^*}
\newcommand{\hatl}{\hat{l}}
\newcommand{\checkl}{\check{l}}

\title[Classical solutions of gas-dynamic equations]{
A kinetic transport-projection splitting algorithm for an hierarchy  of moment closures of gas-kinetic equations}
\author{Misha Perepelitsa}

\thanks{Email: misha@math.uh.edu}

\date{\today}
\address{Misha Perepelitsa\\
misha@math.uh.edu\\
University of Houston\\
PGH 631\\
4800 Calhoun Rd. \\
Houston, TX\\
USA}

\begin{document}
\begin{abstract}
We review some geometrical properties of models of moment closures of gas-kinetic equations,
and consider a transport-projection  splitting scheme for construction of solutions of such closures. The scheme, 
formulated in terms of a dual kinetic density, defines the
kinetic density in successive superposition of transport in $x$--direction and projection  to a finite dimensional linear space in a weighted $L^2$ space, in the kinetic variable $v.$  Given smooth initial data, we show that  the approximate solutions  converge to a unique classical solution of a system of moment closure PDEs. 

\end{abstract}

\maketitle

\begin{section}{Introduction}
\subsection{Motivation}
In a kinetic description of fluid motion the state of the gas is defined by a kinetic function $f(x,t,v),$ that determines the distribution of molecules at position $x$ and time $t$ according to the velocity $v,$ and a kinetic equation for $f,$
\begin{equation}
\label{eq:Kinetic}
\partial_t f{}+{}v\cdot\grad_x f{}={}Q(f),
\end{equation}
where the right-hand side determines the changes in the kinetic density due to molecular interactions. In the kinetic models of gases, the collision operator $Q(f)$ verifies the following properties:
\begin{enumerate}
\item $Q$ has zero moments:
\[
\int(1,v,|v|^2)Q(f)\,dv{}={}0;
\]
\item $Q(f){}={}0$ iff $f\in E_0,$ where $E_0$ is the set of minimizers of the problem
\begin{equation*}
\label{prob:min}
\min\left\{ S(f)\,:\, f\geq0,\quad \int(1,v,|v|^2)f\,dv{}={}const.\right\}
\end{equation*}
with an entropy functional
\[
S(f){}={}\int s(f)\,dv,
\]
were $s\,:\,\mathbb{R}\to\mathbb{R}$ is a convex, coercive function;
\item interactions do not increase entropy: for any kinetic density $f,$
\[
\int Q(f)s'(f)\,dv\leq0.
\]
\end{enumerate}
The above properties are due to the conservation of mass, momentum, and energy
of molecular motion, and express the fact that molecular interactions have an effect on the kinetic density to ``relax" toward the set of equilibrium densities $E_0.$

Assuming that relaxation processes are instantaneous, the kinetic density takes values $f(x,t,\cdot)\in E_0,$ for all $(x,t),$ i.e, it is a function only of its $(1,v,|v|^2)$ moments, that we denote by $U{}={}(\rho,m,E)$ -- the macroscopic density, momentum and energy. The moments verify the system of Euler equations,
\begin{equation}
\label{eq:Euler}
\int \left(\partial_t f{}+{}v\cdot\grad_x f\right)\left(\begin{array}{c}1 \\ v\\|v|^2\end{array}\right)\,dv{}={}0,
\end{equation}
which should be supplemented with the initial/boundary conditions for a particular fluid flow in question.  System of equations \eqref{eq:Euler} is a first order quasi-liner system of PDEs:
\begin{equation}
\label{eq:Hyperbolic}
\partial_t U{}+{}\div_x F(U){}={}0,
\end{equation}
where  the flux $F\,:\, \mathbb{R}^5\to\mathbb{R}^{5\times 3}.$

The conservation of entropy at the kinetic level ($Q(f){}={}0$) leads (for smooth kinetic functions) to the conservation of macroscopic entropy
\[
\partial_t S{}+{}\div_x Q_e{}={}0,
\]
where
\[
S(U){}={}\int s(f)\,dv,\quad Q_e(U){}={}\int vs'(f)\,dv.
\]

In flows away from vacuum, $\rho>0,$ the entropy is a strictly convex function of $U,$
and that makes system \eqref{eq:Hyperbolic} to be {\it symmetrizable, hyperbolic} system of conservation laws. The Cauchy problem for such systems is well-posed in classes of smooth functions, as was established in \cite{Garding, Kato}. Specifically, the following result holds, theorem 5.1.1 of \cite{Dafermos}.

Consider a symmetrizable, hyperbolic system of $m$ conservation laws
\begin{equation}
\label{eq:Hyperbolic:2}
\partial_t U{}+{}\div_x F(U){}={}G(U),
\end{equation}
with $F\in C^4(O)^{m\times d},$ $G\in C^3(O)^m,$ and  entropy $S\in C^3(O),$ on an open subset $O\subset\mathbb{R}^d.$ Assume that $S$ is strictly convex on $O$ and the initial data $U_0(x)$ belong to a compact subset of $O,$ with 
\[
\grad U_0\in H^l,\quad l>d/2.
\]
\begin{theorem*}
There is a time interval $[0,T),$ on which  system of equations \eqref{eq:Hyperbolic:2} with initial data $U_0$ has a unique classical solution $U(x,t).$ Solution belong to the class
\[
\grad U{}\in{}\cap_{k=0}^{l}C^k([0,T): H^{k-l}).
\]
\end{theorem*}
The solution is constructed by a fixed point of a map, determined by a solution of the linearized equations \eqref{eq:Hyperbolic:2}. Alternatively, the solution can obtained in a zero viscosity limit, using the theory of parabolic systems.

In the analysis of non-equilibrium flows, it might be desirable to approximate equation \eqref{eq:Kinetic}  by a closed system of PDEs for a finite set of macroscopic parameters (moments).  A moment closure is an example of such reduction of dimension, which is based  on an ubiquitous idea of Galerkin approximation. A generic form of a moment closure was described in \cite{Gorban1, Gorban2}, and \cite{Levermore}. The moment closures of \cite{Grad:moments, Dreyer} are earlier, notable examples of such approximations. 

Let us consider moment closures is some detail, following \cite{Levermore} for the presentation. The closures are taken with respect to the moments 
\[
\int fl_i(v)\,dv,\quad i=1..k,
\]
where $\{l_i\}_{i=1}^k$ is a set ``elements".  Given the tendency of the kinetic density to an equilibrium in $E_0,$ it is reasonable to include polynomials $\{1,v,|v|^2\}$ among moments. Further restrictions on the set of moments can be imposed by requirement that  the corresponding  system of PDEs \eqref{eq:Hyperbolic:2} verifies Galilean and rotational symmetries, see \cite{Levermore}.

Denote
\[
E^*{}={}\span\{l_i,\,i=1..k\}.
\]
Consider a minimization problem
\begin{equation}
\label{minimization}
\min\left\{ S(f)\,:\, \int fl_i(v)\,dv{}={}U_i=const.\right\}
\end{equation}
where the entropy functional is as above. For a given vector $U{}={}(U_0,..,U_k),$
the problem (typically) has a unique minimizer $f_0,$, determined by the conditions
\begin{eqnarray*}
\exists g_0\in E^*,\quad g_0\in\partial S(f_0),\\
\int f_0 l_i(v)\,dv{}={}U_i,\quad i=1..k.
\end{eqnarray*}
The first condition defined $g_0{}={}s'(f_0)\in E^*,$ or $f_0{}={}(s^*)'(g_0),$ where $s^*$ is the Legendre transform of $s.$\footnote{ We always assume the duality pairing between functions to be $
\langle f,g \rangle{}={}\int fg\,dv. $ A functional  space $V$ for kinetic density $f$ can be define 
\[
V{}={}\left\{f\,:\, \int (1+|v|^{m_0})|f|\,dv<+\infty\right\},
\]
where $m_0$ is the highest degree of polynomials in the set $\{l_i\}.$ In this section we proceed informally, identifying, for example, a subdifferential $\partial S(f)$ with a function $s'(f).$ 
}

Define $E$ to be a set of minimizers $f_0$ for all choices of moments vector $U.$
A moment closure of  \cite{Levermore} is  defined as a system of equations
\begin{equation}
\label{eq:Moment_closure:1}
\begin{cases}
\int \left( \partial_t f{}+{}v\cdot\grad_x f- Q(f)\right)l_i(v)\,dv{}={}0,\quad i=1..k,\\
f(x,t,\cdot)\in E,
\end{cases}
\end{equation}
which is equivalent to a first order quasi-linear system of PDEs of type \eqref{eq:Hyperbolic:2} in unknowns $U_i(x,t){}={}\int f(x,t,v)l_i(v)\,dv.$ The corresponding entropy equation reads:
\[
\partial_t\int s(f)\,dv{}+{}\div_x\int vs'(f)\,dv{}={}\int Q(f)s'(f)\,dv.
\]

Following \cite{Levermore}, the convexity of $S=S(U)$ can be conveniently expressed using dual, hydrodynamic variables $\alpha_i=\alpha_i(x,t),$ related to $f$ by the condition
\[
s'(f(x,t,v)){}={}\sum_i \alpha_i(x,t)l_i(v)\in E^*.
\]
Indeed, the definition of $U$ is stated as
\begin{equation}
\label{eq:Uj}
U_j{}={}\int fl_j(v)\,dv{}={}\int (s')^{-1}(\sum_i\alpha_il_i(v))\,dv.
\end{equation}
Since $S(U){}={}\int s(f)\,dv$ we obtain that
\[
\partial_{U_j}S{}={}\int s'(f)\partial_{U_j}f\,dv{}={}\int (\sum_i\alpha_il_i)\partial_{U_j}f\,dv{}={}\alpha_j.
\]
Thus, $\partial^2_{U_jU_k}S{}={}\partial_{U_k}\alpha_j.$ Also, from \eqref{eq:Uj}
we also get 
\[
\partial_{\alpha_i}U_j{}={}\int\frac{l_j(v)l_i(v)}{s''(\sum\alpha_il_i)}\,dv,
\]
which makes $\partial_\alpha U$ a positive definite. Its inverse is positive definite as well, and so is $\grad^2_{U}S.$

\subsubsection{Orthogonality in primal variables}
An alternative way to derive system \eqref{eq:Moment_closure:1} is to use the differential structure of sets appearing in the optimization problem \eqref{minimization}, following the approach of \cite{Gorban1, Gorban2}.

Consider a kinetic density $f$ and a set of density with the same $l_i$ moments:
\[
M(f){}={}\left\{ \tilde{f}\,:\, \int (\tilde{f}-f)l_i(v)\,dv{}={}0,\,i=1..k\right\}.
\]
Let $f_0\in M_f\cap E_0$ be the minimizer of problem \eqref{minimization}.
Define the tangent plane to $M_f$ at $f_0$ as
\[
T_{M(f)}=\left\{\hat{f}\,:\, \int \hat{f}l_i(v){}={}0,\,i=1..k\right\}.
\]
Since it is independent of $f$ we simply write $T_M.$ The set $E,$ defined above, can be defined as 
\[
f_0\in E\iff \int s'(f_0)\hat{f}\,dv{}={}0,\quad \forall \hat{f}\in T_{M}.
\]
With this definition we can define the tangent space to $E$ at $f_0,$ denoted by $T_E(f_0),$ as 
the set of vectors $\bar{f}$ such that 
\[
\forall \hat{f}\in T_{M},\quad \lim_{t\to0}t^{-1}\int s'(f_0+t\bar{f})\hat{f}\,dv{}={}0.
\]
The condition for being an tangent vector can be equivalently stated as
\begin{equation}
\label{eq:ortho:1}
\int s''(f_0)\hat{f}\bar{f}\,dv{}={}0,\quad \forall \hat{f}\in T_M.
\end{equation}
One can interpret this condition as orthogonality between the tangent spaces to $M$ and $E,$ in a weighted $L^2$ space with scalar product
\[
(\hat{f},\bar{f}){}={}\int \hat{f}\bar{f}\,s''(f_0)dv.
\]
In this notation, $T_M{}={}\left(T_E(f_0)\right)^{\perp}.$ Now, the moment closure system \eqref{eq:Moment_closure:1} can be  equivalently expressed as a differential inclusion
\begin{equation}
\label{eq:Levermore}
\begin{cases}
\partial_tf +v\cdot\grad_x f-Q(f)\in \left(T_E(f_0)\right)^{\perp},\\
 f(x,t,\cdot)\in E,
\end{cases}
\end{equation}
or, by noticing that $\partial_t f\in T_E(f),$ as an equation
\begin{equation}
\label{eq:Gorbin}
\begin{cases}
\partial_tf{}={} \proj_{T_E(f)}(-v\cdot\grad_x f+Q(f)),\\
 f(x,t,\cdot)\in E.
\end{cases}
\end{equation}
The later equation is in the form used in \cite{Gorban1, Gorban2}.

\subsubsection{Orthogonality in dual variables}
Finally, lets consider yet another way to pose a moment closure, expressing \eqref{eq:Levermore} is dual kinetic variables
\[
l{}={}s'(f),\quad (l\in\partial S(f)).
\]
Condition \eqref{eq:ortho:1} carries over to $l$--variables and becomes,
\begin{equation}
\label{eq:ortho:1.5}
\int (s^*)''(l_0)\hat{l}\bar{l}\,dv{}={}0,\quad \forall \hat{l}\in E^*,\,\hat{l}\in T_{M^*(l_0)}.
\end{equation}
It expresses the orthogonality of linear space $E^*$ and tangent space to $M^*(l_0){}={}\{l\,:\, \int ((s^*)'(l)-(s^*)'(l_0))l_i(v)\,dv{}={}const.\}$

Orthogonality condition \eqref{eq:ortho:1.5} has another interpretation. Recall from the Convex Analysis, \cite{Ekland}, proposition 2.4, that the values of the primal problem
\[
\min\left\{ S(f)\,:\, f\geq0,\,\int(f-\check{f})l_i(v)\,dv{}={}0,\,\,i=1..k\right\}
\]
(for a fixed $\check{f}\geq0,$) and its dual
\begin{equation}
\label{prob:dual}
\sup\left\{ \int l\check{f}\,dv{}-{}S^*(l)\,:\,l\in E^*\right\}
\end{equation}
where $S^*(l){}={}\int s^*(l)\,dv$ is the Legendre transform of $S(f),$ are equal. The minimizer $f_0$ and the maximizer $l_0$ are determined by the conditions
\begin{eqnarray}
l_0\in\partial S(f_0),\quad (l_0{}={}s'(f_0),\,f_0{}={}(s^*)'(l_0)), \notag \\
\label{cond:moments}
\int (f_0-\tilde{f})l_i(v)\,dv{}={}0,\quad i=1..k.
\end{eqnarray}

Choose  $h$ that verifies condition \eqref{eq:ortho:1.5}: 
\begin{equation*}
\int (s^*)''(l_0)hl_i\,dv{}={}0,\quad i=1..k,
\end{equation*}
and let  $l_0\in E^*.$ Consider the above maximization problem with $\tilde{f}{}={}
(s^*)'(\tilde{l}),$ $\tilde{l}{}={}l_0+th.$ Let $l^{t}_0$ be the corresponding maximizer. Re-writing condition \eqref{cond:moments} as
\[
\int ( (s^*)'(l_0+th)-(s^*)'(l^{t}_0))l_i(v)\,dv{}={}0.
\]
we see that due to assumptions on $h,$
\[
l^{t}_0-l_0{}={}o(t),\quad l_0{}={}\proj_{E^*}(\tilde{l}),
\]
where the projection with respect ot weighted $L^2$ norm. In other words, the solution of the optimization problem \eqref{prob:dual} with $\tilde{f}{}={}
(s^*)'(\tilde{l}),$ coincides (to the first order of distance from $E^*$) with the projection of $\tilde{l}$ onto $E^*.$

This considerations allow us to re-write the moment closure equations in \eqref{eq:Levermore} as  a differential inclusion in dual variable $l:$
\begin{equation}
\label{eq:Misha}
\begin{cases}
\partial_t l{}+{}v\cdot\grad_x l{}-{}\tilde{Q}(l)\in (E^*(l))^{\perp},\\
 l(x,t,\cdot)\in E^*,
\end{cases}
\end{equation}
where the collision operator equals
\[
\tilde{Q}(l){}={}\frac{Q((s^*)'(l))}{(s^*)''(l)}.
\]
We certainly could have arrived at \eqref{eq:Misha} directly from \eqref{eq:Levermore}, but the above arguments  show that there is also an underlying variational principle. 

%As long as only  smooth  kinetic densities are considered, all formulations \eqref{eq:Moment_closure:1}, \eqref{eq:Levermore}, \eqref{eq:Gorbin} and \eqref{eq:Misha} are equivalent. They equivalent to a system of PDEs \eqref{eq:Hyperbolic:2} in hydrodynamics variables $U(x,t).$

Let us remark, that the  dual kinetic variables and weighted $L^2$ spaces, discussed above, has been in use in the theory of Boltzmann equations since the work of Hilbert\cite{Hilbert}, where they appear in a context of  linearization  of  \eqref{eq:Kinetic}. In a typical linearization analysis, kinetic density is represented in terms of a dual variable $h,$ as $f{}={}f_0(1+h),$ where $f_0$ is a Maxwellian ($f_0\in E_0$).

\subsection{Results}
In this work we establish the existence of classical solutions  to a class of systems of PDEs \eqref{eq:Hyperbolic:2} corresponding to \eqref{eq:Misha}, by solving the later problem in a space of kinetic functions. We assume that initial data $l^0(x,\cdot)$ take values in $E^*,$ ranging in a neighborhood of a constant state $\bar{l}\in E^*,$ and, is in Sobolev's $H^3$ space, as a function of $x.$

%Our approach privides an alternative proof of existence of solutions to a class of pde's using the structure of  differntial inclusion  \eqref{eq:Misha}, rather than then the structure of symmetrizable, hyporbolic pde's.

Two types of collision operators are considered. In the first model, the collision operator is absent, $\tilde{Q}=0,$ and we are dealing with projection of a transport equation onto $\E.$ In the second model, we consider a non-linear  BGK--type operator, in dual variable $l:$
\[
\tilde{Q}(l){}={}\Pi^0_{l}-l,\quad \Pi^0_l{}={}\proj_{E_0^*(l)}(l).
\]
The corresponding operator in the primal variable, $Q(f){}={}(s^*)''(l)\tilde{Q}(l),$ $l{}={}(s)'(f),$
verifies properties (1), (2), and (3) of the collision operators, stated at the beginning of the Introduction. It is a first order (in the distance from $E_0$)  approximation of the classical BGK operator.

We show that classical solutions of \eqref{eq:Misha} can be constructed in zero limit of step $h,$ of a time-discretization of \eqref{eq:Misha}, in which transport, collision and projection are computed in succession, over time intervals $(nh, (n+1)h],$ as defined in \eqref{Projection}, \eqref{approx:continuous:1}, \eqref{Projection:2}.

The analysis is based on entropy estimates for the kinetic density and its $x$--derivatives, that  are similar to the estimates for linearized Boltzmann equation  obtained in \cite{Grad:asymptote}.

%It is best explained on an example of a  simple linear equation,
%\[
%\partial_t l{}+{}v\cdot\grad_xl{}={}\frac{\proj_{E^*}(l)-l}{h},
%\]
%where $E^*$ is a closed, linear subspace of  $V^*{}={}L^2_w(\mathbb{R}^3),$ with repect %to a weight $w=w(v)\geq0, $ $w\in C^\infty_0(\mathbb{R}^3).$ 
%The entropy estimates for $l$ and its spacial derivatives are easily obtained; 
%\begin{equation*}
%\frac{d}{dt}\iint (\sum_\alpha|D^\alpha_x l|^2)w\,dvdx {}+{}
%\frac{1}{h}\iint (\sum_\alpha
%|\proj_{E^*}(D^\alpha_xl)-D^\alpha_x l|^2)w\,dvdx{}={}0, 
%\end{equation*}
%for any values of  multi-index $\alpha.$ 
%
% The uniform in $h$ estimates for $l$ and its derivatives $D^\alpha_x l$ follow, as %well as convergence of $l^h(x,t,\cdot)$ to an element of $E^*,$ as $h\to0.$ The %estimates are analogous of those for linearized Boltzmann equation, %\cite{Grad:asymptote}. For a non-linear problem \eqref{eq:Misha}, we use a %time-%descrete version of the above estimates, accounting for the dependence  the %weight function on $l.$ 

The choice of the approximating scheme is not accidental. In fact, it is the convergence of that particular scheme that we're interested in, rather than finding a new way to prove the existence of classical solutions for a class of PDEs \eqref{eq:Hyperbolic:2}. The reason for this, is an observation that the time discretization of \eqref{eq:Misha}, with $E^*{}={}E_0^*,$ and $\tilde{Q}{}={}0,$  is linked to a hydrodynamic limit of a gas-kinetic equation \eqref{eq:Kinetic} with the right-hand side containing a large factor $h^{-1}.$ The projection of the kinetic density to an equilibrium $E_0$ can be loosely related to result of collisions, since the later amounts to relaxing the density toward equilibrium, while conserving the moments.  

In this respect, our convergence result can be compared with the works on the fluid dynamic limit of Boltzmann and related equations, with some representative results given in  \cite{Nishida_Boltzmann, Caflisch, CaflischPapa, BB, BV}. 

In the present setting, the convergence takes place for all times $t$ inside an interval $[0,T]$ determined by the initial data. No initial layer is present since the dynamics is smooth and starts from the target  manifold $E^*.$

Finally, let us mention that discrete transport--projection 
approximations appear in many areas of PDEs. Some examples of the method, in the context of Boltzmann equation and scalar conservation laws can be found in \cite{Gorban3, Brenier1, Brenier2}.

\end{section}

%%%%%%%%%%%%%%%%%%%%%%%%%%%%%%%%%%%%%%%%%%%%%%%%%%%%%%%
%%%%%%%%%%%%%%%%%%%%%%%%%%%%%%%%%%%%%%%%%%%%%%%%%%%%%%%
%%%%%%%%%%%%%%%%%%%%%%%%%%%%%%%%%%%%%%%%%%%%%%%%%%%%%%%

\begin{section}{Transport equation}

\subsection{Notation and auxiliary lemmas}

Let $\{l_i(v)\}_{i=1}^{k}$ be a set of $k+1$ linearly independent on an open set of $\mathbb{R}^d$ polynomials. We assume that the highest degree polynomial is $l_k(v){}={}|v|^{m_0},$ for some $m_0>0,$ and the lowest degree polynomial $l_1(v){}={}1.$
Denote
\[
\E{}={}\span\{l_i,\:\,i=1..k\}.
\]

We choose entropy density to be $s(f){}={}f^p,$ with $p\in(1,6/5).$ The Legendre transform of $s$ equals 
\[
s^*(l){}={}c_p(l_+)^{p/{p-1}},
\]
with $c_p{}={}\frac{p-1}{p^{p/(p-1)}},$
and  $l_+$ is a positive part of $l.$

For a notational convenience we define the weight function with respect to variable $-l,$ rather than $l,$ 
\[
w(l){}={}\bar{c}_p(l_{-})^{\frac{2-p}{p-1}},
\]
where $\bar{c}_p>0,$ and $l_{-}\geq0$ -- the negative part of $l.$ 
For the range of $p$ defined above, $w\in C^4(\mathbb{R}),$ convex function, supported on   $l\geq0.$ With the above choice of an entropy $s$, kinetic densities $f\in E,$ are smooth and compactly supported functions. The analysis critically depends on last two properties.

We consider a Cauchy problem
\begin{equation}
\label{model:1}
\begin{cases}
\partial_t l {}+{}v\cdot\grad_x l{}\in{} (E^*(l))^\perp, & (x,t,v)\in \mathbb{R}^3\times(0,T)\times\mathbb{R}^3,\\
l(x,t,\cdot)\in E^*, & (x,t)\in\mathbb{R}^4_+, \\
l(x,0,v){}={}l^0(x,v), & (x,v)\in\mathbb{R}^6,
\end{cases}
\end{equation}
where notation $\E(l)$ denotes space $\E$ with weighted $L^2$ norm
\[
\|\tilde{l}\|^2_{w(l)}{}={}\int \tilde{l}^2\,w(l)\,dv.
\]

Following \cite{Dafermos},
 we make the following definition.
\begin{definition}
A classical solution $l,$ of \eqref{model:1}
as a Lipschitz continuous in $(x,t,v)$ function $l(x,t,\cdot),$ such that for a.e. $(x,t),$ $t\geq0,$ the set of $k+1$ equations holds,
\[
\int \left( \partial_t l{}+{}v\cdot\grad_x l\right)l_i(v)\,dv{}={}0,\quad i=1..k;
\]

for all $(x,t),t\geq0,$ $l(x,t,\cdot)\in \E;$  for $t=0,$ and all $(x,v),$ $l(x,0,v){}={}l^0(x,v).$ 
\end{definition}

A care should be taken to avoid degenerate situation when the weight $w(l)$ is zero.   
\begin{definition} Let $(R,\delta_1,r,\delta_2)$ be positive numbers. We say that $l\in \E$ has property $P(R,\delta_1,r,\delta_2)$ if
\begin{enumerate}
\item $\forall v,$ with $|v|>R,$ $
l(v)\geq0;
$
\item $\forall v,$ with $|v|<R,$ 
$ l(v)\geq -\delta_1;
$
\item there is a ball $B_r$ of radius $r$ such that 
$
  l(v)\leq -\delta_2,
$
for all $v\in B_r.$
\end{enumerate}
\end{definition}
Let $\barl{}={}\sum_i\bar{\gamma}_il_i(v)\in \E$ be a constant reference state with 
\[
\bar{\gamma}_0<0,\quad \bar{\gamma}_k>0.
\]
The solutions we construct  are in a neighborhood of  $\bar{l}.$ The following lemma is easily verified.
\begin{lemma} 
\label{lemma:nondegenerate}
There are positive numbers $(R,\delta_1,r,\delta_2)$ such that 
\[
\barl\in P\left(R/2,\delta_1/2,2r,2\delta_2\right),
\]
and 
$\forall \e>0$ there is $\Delta>0,$ such that, if $l\in \E$ and
\[
\int_{B_R}|l-\barl|^2\,dv\leq \Delta,
\]
then
\[
l\in P\left(\frac{1+\e}{2}R,\frac{1+\e}{2}\delta_1, \frac{2}{1+\e}r,\frac{2}{1+\e}\delta_2\right).
\]
\end{lemma}
Let $\Delta_2$ be a number corresponding to $\e=1,$ and $\Delta_1<\Delta_2$ be the number corresponding to $\e{}={}1/2,$ in the above lemma.
The numbers $(R,\delta_1,r,\delta_2)$ and the corresponding balls $B_R,$ $B_r$ from the definition of the property $P,$ will be fixed in the following analysis.

Let the initial date $l^0(x,v)$ be such that
\begin{equation}
\label{cond:initial}
\begin{cases}
\forall x\in\mathbb{R}^3, \quad l^0(x,\cdot)\in \E,\\
l^0-\barl\in L^2(B_R; H^3(\mathbb{R}^3)), \\
\sup_x\int_{B_R}|l^0(x,v)-\barl(v)|^2\,dv\leq \Delta_1.
\end{cases}
\end{equation}
We use a weighted ``norm" for $l,$ defined as:
\begin{equation}
\label{X}
\|l\|_{X(l)}{}={}\left(\iint (\sum_{\alpha,|\alpha|\leq 3}|D^\alpha_x l|^2)w(l)\,dvdx\right)^{1/2}.
\end{equation}

\subsection{Statement of the result}
The discrete-time algorithm approximating differential inclusion \eqref{model:1} is defined in the following way.
Let $h>0$ be the time step and define $N{}={}\lceil T/h\rceil.$ Given the values of $l^{n-1},$ we define 
\begin{equation}
\label{Projection}
l^n(x,\cdot){}={}\proj_{\E(l^{n-1}(x,\cdot))}l^{n-1}(x-h\cdot,\cdot),
\end{equation}
where the projection is in weighted $L^2$ space with weight $w(l^{n-1}).$ If the weight is not zero identically, the projection is uniquely defined by conditions
\begin{equation}
\label{Projection_cond}
\forall x\in\mathbb{R}^3,\,i=1..k,\quad 
\int l^n(x,v)l_i(v)w(l^{n-1}(x,v))\,dv{}={}
\int l^{n-1}(x-hv,v)l_i(v)w(l^{n-1}(x,v))\,dv.
\end{equation}
In what follows we use the shorthand notation $w^n(x,v){}={}w(l^n(x,v)).$

\begin{theorem}
\label{theorem:1}
Let $\barl,l^0$ be as described above and $h\in(0,1].$ There is time $T>0,$ independent of $h,$ such that all functions $l^n,$ $n=0..\lceil T/h\rceil,$ in \eqref{Projection} are well defined, and  the family $\{l^h\}$ of interpolations of $l^n$'s, defined in \eqref{approx:continuous:1}, converges as $h\to0$, uniformly on compact set in $\mathbb{R}^3\times[0,T]\times\mathbb{R}^3$ to a unique classical solution of \eqref{model:1}.
\end{theorem}
\begin{remark}
The uniqueness of classical solutions (in fact strong-weak uniqueness) follows from the uniqueness of classical solutions of corresponding PDEs \eqref{eq:Hyperbolic:2}, with strictly convex entropy, see, for example, \cite{Dafermos}, theorem 5.3.1.
\end{remark}

The theorem is based on the fact that functions $l^n$ are bounded in strong norms, which can be heuristically explained as follows. Suppose that the approximation $\{l^n\}$ is well-defined. Denote by $\hatl^{n-1}(x,v){}={}l^{n-1}(x-hv,v),$ and for any $x\in\mathbb{R}^3,$ the distance
\[
|l(x,\cdot)-\barl(\cdot)|^2_{w^{n}}{}={}\int |l(x,v)-\barl(v)|^2w(l^n(x,v))\,dv.
\]
From \eqref{Projection} we obtain
\[
|l^n-\barl|^2_{w^{n-1}}\leq |\hatl^{n-1}-\barl|^2_{w^{n-1}}.
\]
We will show that
\begin{equation}
\label{est:zero}
|l^n-\barl|^2_{w^n}{}\leq{}|\hatl^{n-1}-\barl|^2_{w(\hatl^{n-1})}{}+{}R_n,
\end{equation}
where $R_n$ accounts for changes in the weights from $w^n$ to $w^{n-1},$ and from  $w(l^{n-1})$ to $w(\hatl^{n-1}).$ The remainder is such that 
\begin{equation}
\label{R1}
\int R_n\,dx{}={}O(h),
\end{equation}
provided that all $l^n$ are smooth, as measured by \eqref{X}. Integrating \eqref{est:zero} we obtain
\begin{equation}
\label{est:one}
\int|l^n-\barl|^2_{w^n}\,dx{}\leq{}\int|l^{n-1}-\barl|^2_{w^{n-1}}\,dx{}+{}O(h).
\end{equation}

To estimate the spacial derivatives we use an orthogonal decomposition (in topology of $L^2_{w^{n-1}}(\mathbb{R}^3)$)
\[
(l^n-\barl){}+{}(\hatl^{n-1}-l^n){}={}\hatl^{n-1}-\barl,
\]
to obtain 
\[
D^\alpha_x(l^n-\barl){}+{}D^\alpha_x(\hatl^{n-1}-l^n){}={}D^\alpha_x(\hatl^{n-1}-\barl),
\]
which implies that
\[
|D^\alpha_x(l^n-\barl)|^2_{w^{n-1}}{}\leq{}|D^\alpha_x(\hatl^{n-1}-\barl)|^2_{w^{n-1}}{}+{}R^\alpha_n,
\]
where $R^\alpha_n$ accounts for spacial derivatives of the weight function, and has property \eqref{R1}. From this, by changing the weights, 
\[
|D^\alpha_x(l^n-\barl)|^2_{w^{n}}{}\leq{}|D^\alpha_x(\hatl^{n-1}-\barl)|^2_{w(\hatl^{n-1})}{}+{}\tilde{R}^\alpha_n,
\]
which leads to 
\begin{equation}
\label{est:two}
\int|D^\alpha_x(l^n-\barl)|^2_{w^n}\,dx{}\leq{}\int|D^\alpha_x(l^{n-1}-\barl)|^2_{w^{n-1}}\,dx{}+{}O(h).
\end{equation}
With derivatives of order 3, \eqref{est:one}, \eqref{est:two} lead to a priori estimates on $l^n$ as measured by $X(l^n).$ 

The above arguments are formalized in lemmas \ref{lemma:1}--\ref{lemma:Final} below, after which we show that properly interpolated in time sequence $l^n$ converges to a classical solution of \eqref{model:1}.

%%%%%%%%%%%%%%%%%%%%%%%%%%%%%%%%%%%%%%%%%%%%%%%%%%%%%%%
%%%%%%%%%%%%%%%%%%%%%%%%%%%%%%%%%%%%%%%%%%%%%%%%%%%%%%%
%%%%%%%%%%%%%%%%%%%%%%%%%%%%%%%%%%%%%%%%%%%%%%%%%%%%%%%

\subsection{Proof of theorem \ref{theorem:1}}
Consider the sequence $\{l^n\}$ determined from the initial data and \eqref{Projection}. We will assume in this section the following hypothesis.
\begin{hypothesis} For all $n=1..N,$ and all $x\in\mathbb{R}^3,$
\[
l^n(x,\cdot)\in P(R,\delta_1,r,\delta_2).
\]
\end{hypothesis}
Under this hypotheses, functions $l^n=\sum_i\gamma_i^n(x)l_i(v)$ from \eqref{Projection} are well-defined and we proceed to derive energy estimates.

\begin{lemma}
\label{lemma:1}
There is $C>0,$ independent of $(n,h),$ such that
\begin{equation}
\label{est:1}
\sup_x \int_{B_R}|l^{n-1}(x-hv,v)-l^{n-1}(x,v)|^2\,dv\leq Ch^2\|l^{n-1}-\barl\|^2_{X(l^{n-1})}.
\end{equation}
\end{lemma}

\begin{proof} The estimate follows directly by applying lemma \ref{Sobolev:1} and  lemma \ref{Sobolev:2} from the Appendix to function $l^{n-1}.$
\end{proof}

\begin{lemma}
There is $C>0,$ independent of $(n,h),$ such that
\begin{equation}
\label{est:2}
\sup_x\int_{B_R}|l^{n}(x,v)-l^{n-1}(x,v)|^2\,dv\leq Ch^2\|l^{n-1}-\barl\|^2_{X(l^{n-1})},
\end{equation}
and 
\begin{equation}
\label{eneqry:00}
\int \int_{B_R} |l^n-l^{n-1}|^2\,dvdx\leq Ch^2\|l^{n-1}-\barl\|^2_{X(l^{n-1})}.
\end{equation}
\end{lemma}
\begin{proof}
Indeed, using \eqref{Projection} we obtain
\[
\sup_x\int|l^{n}(x,v)-l^{n-1}(x,v)|^2w^{n-1}\,dv\leq C\sup_x\int_{B_R}|l^{n-1}(x-hv,v)-l^{n-1}(x,v)|^2\,dv,
\]
and the first statement of the lemma follows from \eqref{est:1}, and the facts that $w^{n-1}$ is strictly positive on the ball $B_r,$ and norms with respect to balls $B_r$ and $B_R$ are equivalent.
Similarly, from conditions \eqref{Projection} and Hypothesis 1 we get
\[
\iint |l^n-l^{n-1}|^2w^{n-1}\,dvdx\leq C\int_{B_R}\int |l^{n-1}(x-hv,v)-l^{n-1}(x,v)|^2\,dxdv
\]
which is less than
$
Ch^2\|l^{n-1}-\barl\|^2_{X(l^{n-1})},
$
by estimate \ref{est:1}.
\end{proof}

\begin{lemma}
\label{lemma:3}
 There is $C>0,$ independent of $(n,h),$ such that
\[
\int\int_{B_R} |D^\beta l^n-D^\beta l^{n-1}|^2\,dxdv{}\leq{}Ch^2(\|l^{n-1}-\barl\|^2_{X(l^{n-1})}{}+{}
\|l^{n-1}-\barl\|^4_{X(l^{n-1})}),
\]
for any multi-index $\beta,$ with $|\beta|=1;$
\[
\int\int_{B_R} |D^\beta l^n-D^\beta l^{n-1}|^2\,dxdv{}\leq{}Ch^2(\|l^{n-1}-\barl\|^2_{X(l^{n-1})}{}+{}
\|l^{n-1}-\barl\|^4_{X(l^{n-1})}{}+{}\|l^{n-1}-\barl\|^6_{X(l^{n-1})}),
\]
for any multi-index $\beta,$ with $|\beta|=2.$
\end{lemma}
\begin{proof} Let $D$ be generic notation for the first derivative in $x.$ By applying it to \eqref{Projection} we find that
\begin{eqnarray}
\label{diff:1}
\int(Dl^n - Dl^{n-1})l_iw^{n-1}\,dv &=& \int (Dl^{n-1}(x-hv,v)-Dl^{n-1}(x,v))l_iw^{n-1}\,dv\\
&&+\int (l^{n-1}(x-hv,v)-l^n(x,v))l_iw'(l^{n-1})Dl^{n-1}\,dv. \notag
\end{eqnarray}
It follows that for any $x,$
\begin{eqnarray*}
\int_{B_R}|Dl^n- Dl^{n-1}|^2\,dv{}&\leq&{} C\int_{B_R}|Dl^{n-1}(x-hv,v)-Dl^{n-1}(x,v)|^2\,dv\\
&&{}+{}
C\int_{B_R}|l^{n-1}(x-hv,v)-l^{n-1}(x,v)|^2|Dl^{n-1}|^2\,dv\\
&&{}+{}\int_{B_R}|l^n-l^{n-1}|^2|Dl^{n-1}|^2\,dv\\
&\leq& C\int_{B_R}|Dl^{n-1}(x-hv,v)-Dl^{n-1}(x,v)|^2\,dv\\
&&{}+{}\int_{B_R}|Dl^{n-1}|^2\,dv \int_{B_R}|l^{n-1}(x-hv,v)-l^{n-1}(x,v)|^2\,dv\\
&&{}+{}\int_{B_R}|Dl^{n-1}|^2\,dv \int_{B_R}|l^{n}-l^{n-1}|^2\,dv.
\end{eqnarray*}
Using \eqref{est:1}, \eqref{est:2} and lemma \ref{Sobolev:2} we get the first inequality in the lemma.
\[
\int\int_{B_R}|Dl^n- Dl^{n-1}|^2\,dv{}\leq{} Ch^2( \|l^{n-1}-\barl\|^2_{X(l^{n-1})} + \|l^{n-1}-\barl\|^4_{X(l^{n-1})}).
\]
The second inequality is obtained by differentiating \eqref{diff:1} and repeating the arguments above.
\end{proof}

\begin{lemma}[Zero order entropy estimate]
There is $C>0,$ independent of $(n,h),$ such that 
\begin{equation}
\label{energy:0}
\iint |l^n-\barl|^2w^{n}\,dvdx{}\leq {} \iint |l^{n-1}-\barl|^2w^{n-1}\,dvdx+Ch\|l^{n-1}-\barl\|^2_{X(l^{n-1})}.
\end{equation}
\end{lemma}
\begin{proof}
To get the estimate we use \eqref{Projection} to write
\begin{multline*}
\iint |l^n-\barl|^2w^{n-1}\,dvdx{}\leq{}\iint|l^{n-1}(x-hv,v)-\barl|^2w^{n-1}dvdx\\
{}\leq{}
\iint|l^{n-1}(x,v)-\barl|^2w^{n-1}\,dvdx{}+{}\iint|l^{n-1}(x,v)-\barl|^2|w^{n-1}-w(l^{n-1}(x+hv,v))|\,dvdx
\end{multline*}
The last term can be estimated as
\begin{multline*}
\iint|l^{n-1}(x,v)-\barl|^2|w^{n-1}-w(l^{n-1}(x+hv,v))|\,dvdx\\
{}\leq{} C\int\int_{B_R}|l^{n-1}(x,v)-\barl|^2|l^{n-1}(x,v)-l^{n-1}(x-hv,v)|\,dvdx\\
{}\leq{}C\int \sup_{v\in B_R} |l^{n-1}(x,v)-l^{n-1}(x-hv,v)|\int_{B_R}|l^{n-1}(x,v)-\barl|^2\,dv\,dx\\
{}\leq{}C \int \int_{B_R}|l^{n-1}(x,v)-l^{n-1}(x-hv,v)|\,dv\int_{B_R}|l^{n-1}(x,v)-\barl|^2\,dv\,dx\\
{}\leq{}C\left(\sup_x\int_{B_R}|l^{n-1}(x,v)-l^{n-1}(x-hv,v)|\,dv\right)\left(\int\int_{B_R}|l^{n-1}(x,v)-\barl|^2\,dvdx\right)\\
{}\leq{}
Ch\|l^{n-1}-\barl\|^2_{X(l^{n-1})},
\end{multline*}
where we used equivalence of norms in $v,$ and  \eqref{est:1}.
Using this in the previous inequality results in the statement of the lemma.

\end{proof}
Higher order energy estimates are obtained by differentiating conditions \eqref{Projection} and following the arguments of the previous lemma.
\begin{lemma}[Third order entropy estimates]
\label{lemma:energy:3}
There is $C>0,$ independent of $(n,h),$ such that for any multi-index $\alpha,$ with $|\alpha|=3,$
\begin{multline}
\label{energy:3}
\iint |D^\alpha l^n|^2w^n\,dxdv{}\leq{}\iint |D^\alpha l^{n-1}|^2w^{n-1}\,dxdv{}+{}Ch\left( \|l^n-\barl\|^3_{X(l^n)}{}+{}
\|l^n-\barl\|^2_{X(l^n)}\right.\\
\left.{}+{}\|l^{n-1}-\barl\|^2_{X(l^{n-1})}{}+{}\|l^{n-1}-\barl\|^3_{X(l^{n-1})}{}+{}\|l^{n-1}-\barl\|^5_{X(l^{n-1})}\right).
\end{multline}
\end{lemma}

\begin{proof}
After differentiating \eqref{Projection} by $D^\alpha{}={}\partial_{x_1}^{\alpha_1}\partial_{x_2}^{\alpha_2}\partial_{x_3}^{\alpha_3}$ we obtain
\begin{eqnarray*}
\int D^\alpha l^nl_iw^{n-1}\,dv&=&\int D^\alpha l^{n-1}(x-hv,v)l_i w^{n-1}\,dv\\
&&{}-{} \int (l^n-l^{n-1})l_iD^\alpha w^{n-1}\,dv{}-{}\int(l^{n-1}(x,v)-l^{n-1}(x-hv,v))l_iD^\alpha w^{n-1}\\
&&{}-{}\sum_{\beta,\gamma}\int D^\beta(l^n-l^{n-1})l_iD^\gamma w^{n-1}\,dv\\
&&{}-{}\sum_{\beta,\gamma}\int D^\beta(l^{n-1}(x,v)-l^{n-1}(x-hv,v))l_iD^\gamma w^{n-1}\,dv,\quad i=1..k,
\end{eqnarray*}
where multi-indices's $\beta,\gamma$ are such that $\beta+\gamma=\alpha,$ $|\beta|,|\gamma|>0.$
From this we obtain 
\begin{eqnarray*}
\int |D^\alpha l^n|^2w^{n-1}\,dv &\leq& \int|D^\alpha l^{n-1}(x-hv,v)|^2w^{n-1}\\
&&+ 2\int|l^n-l^{n-1}||D^\alpha l^n||D^\alpha w^{n-1}|\,dv\\
&&+ 2\int|l^{n-1}(x,v)-l^{n-1}(x-hv,v)||D^\alpha l^n||D^\alpha w^{n-1}|\,dv\\
&&+2\sum_{\beta,\gamma}\int |D^\beta(l^n-l^{n-1})||D^\alpha l^n||D^\gamma w^{n-1}|\,dv\\
&&+2\sum_{\beta,\gamma}\int |D^\beta(l^{n-1}(x,v)-l^{n-1}(x-hv,v))||D^\alpha l^n||D^\gamma w^{n-1}|\,dv.
\end{eqnarray*}
Labeling the last four terms as $I_1,..,I_4,$ we write
\begin{eqnarray}
\label{l:1}
\iint |D^\alpha l^n|^2w^{n}\,dvdx &\leq& \iint|D^\alpha l^{n-1}(x-hv,v)|^2w(l^{n-1}(x-hv,v))\,dvdx\\
&& +\iint|D^\alpha l^{n-1}(x-hv,v)|^2|w(l^{n-1}(x-hv,v)-w(l^{n-1}(x,v)|\,dvdx \notag \\
&&+ \iint|D^\alpha l^n|^2|w^n-w^{n-1}|\,dvdx{}+{}\sum_j \int I_j\,dx. \notag
\end{eqnarray}
In this way we obtained inequality
\begin{eqnarray*}
\iint |D^\alpha l^n|^2w^{n}\,dvdx &\leq& \iint|D^\alpha l^{n-1}|^2w^{n-1}\,dvdx\\
&&{}+{}\int J_1\,dx+\int J_2\,dx{}+{}\sum_j \int I_j\,dx,
\end{eqnarray*}
where by $J_1,J_2$ we denote the second and the third terms on the right in \eqref{l:1}.
It remains to show that integrals of $J_i's$ and $I_j's$ are of the order $h.$

\begin{eqnarray*}
\int J_1\,dx &\leq&C\int_{B_R}\left(\int |D^\alpha l^{n-1}(x-hv,v)|^2\,dx\right) \sup_x|l^{n-1}(x-hv,v)-l^{n-1}(x,v)|\,dv\\
&\leq& Ch\int_{B_R}\left(\int |D^\alpha l^{n-1}(x-hv,v)|^2\,dx\right) \|l^{n-1}-\barl\|_{H^3(\mathbb{R}^3)}\,dv\\
&\leq& Ch\sup_{v\in B_R}\|l^{n-1}-\barl\|^2_{X(l^{n-1})}{}\leq{}Ch \|l^{n-1}-\barl\|^3_{X(l^{n-1})}.
\end{eqnarray*}

\begin{eqnarray*}
\int J_2\,dx &\leq& C\int_{B_R}\int |D^\alpha l^n|^2|l^n-l^{n-1}|\,dvdx\\
&\leq& C\int \left(\int_{B_R}|l^n-\barl|^2\,dv\right)^{1/2}\int_{B_R}|D^\alpha l^n|^2\,dv\,dx\\
&\leq& C\sup_x \left(\int_{B_R}|l^n-\barl|^2\,dv\right)^{1/2}\|l^n-\barl\|^2_{X(l^n)}{}\leq{}Ch\|l^n-\barl\|^3_{X(l^n)},
\end{eqnarray*}
where in the last inequality we used \eqref{est:2}.

Consider now
\begin{eqnarray*}
\int I_1\,dx &\leq&C\int\int_{B_R}|l^n-l^{n-1}||D^\alpha l^n||Dl^{n-1}|^3\,dvdx
{}+{}C\int\int_{B_R}|l^n-l^{n-1}||D^\alpha l^n||Dl^{n-1}||D^2l^{n-1}|\,dvdx\\
&&+ C\int\int_{B_R}|l^n-l^{n-1}||D^\alpha l^n||D^3l^{n-1}|\,dvdx{}={}K_1+K_2+K_3,
\end{eqnarray*}
where $Dl,D^2l,D^3l$ denote all derivatives in $x$ of orders 1,2, and 3. Then,
\begin{eqnarray*}
K_1&\leq& C\int \sup_{v\in B_R}\left(|l^n-l^{n-1}||Dl^{n-1}|^2\right)\int_{B_R}|D^\alpha l^n|^2+|Dl^{n-1}|^2\,dv\,dx\\
&\leq& C\int \sup_{v\in B_R}|l^n-l^{n-1}|\sup_{v\in B_R}|Dl^{n-1}|^2\int_{B_R}|D^\alpha l^n|^2+|Dl^{n-1}|^2\,dv\,dx\\
&\leq& C\int \left(\int_{B_R}|l^n-l^{n-1}|^2\,dv\right)^{1/2}\int_{B_R}|Dl^{n-1}|^2\,dv\int_{B_R}|D^\alpha l^n|^2+|Dl^{n-1}|^2\,dv\,dx\\
&\leq& C\sup_x \left(\int_{B_R}|l^n-l^{n-1}|^2\,dv\right)^{1/2}\sup_x\int_{B_R}|Dl^{n-1}|^2\,dv\left(\|l^n-\barl\|^2_{X(l^n)}{}+{}\|l^{n-1}-\barl\|^2_{X(l^{n-1})}\right).
\end{eqnarray*}
Using estimates of lemma \ref{Sobolev:1} and \eqref{est:2}, we conclude that
\begin{equation*}
K_1\leq Ch\|l^{n-1}-\barl\|^3_{X(l^{n-1})}\left(\|l^n-\barl\|^2_{X(l^n)}{}+{}\|l^{n-1}-\barl\|^2_{X(l^{n-1})}\right).
\end{equation*}
Estimates on $K_2,K_3$ are similar. They lead to:
\begin{eqnarray}
\label{est:I1}
\int I_1\,dx&\leq&Ch\left(\|l^{n-1}-\barl\|_{X(l^{n-1})}+\|l^{n-1}-\barl\|^2_{X(l^{n-1})}+\|l^{n-1}-\barl\|^3_{X(l^{n-1})}\right)\\
&&{}\times{}\left(\|l^n-\barl\|^2_{X(l^n)}{}+{}\|l^{n-1}-\barl\|^2_{X(l^{n-1})}\right). \notag
\end{eqnarray}
The estimate on $I_2$ is analogous to that of $I_1,$ where we use \eqref{est:1} instead of \eqref{est:1}. It lead to the estimate \eqref{est:I1}.

Consider $I_4:$
\begin{eqnarray*}
\int I_4\,dx&\leq& C\int\int_{B_R}|Dl^{n-1}(x-hv,v)-Dl^{n-1}(x,v)||D^\alpha l^n||D^2l^{n-1}|\,dvdx\\
&&{}+{}C\int\int_{B_R}|Dl^{n-1}(x-hv,v)-Dl^{n-1}(x,v)||D^\alpha l^n||Dl^{n-1}|^2\,dvdx\\
&&{}+{}C\int\int_{B_R}|D^2l^{n-1}(x-hv,v)-D^2l^{n-1}(x,v)||D^\alpha l^n||Dl^{n-1}|\,dvdx{}={}M_1+M_2+M_3.
\end{eqnarray*}

\begin{eqnarray*}
M_3 &\leq &C\sup_{v\in B_R,x\in\mathbb{R}^2}|Dl^{n-1}|\left(\int\int_{B_R}|D^2l^{n-1}(x-hv,v)-D^2l^{n-1}(x,v)|^2\,dvdx\right)^{1/2}\\
&&{}\times{}\left(\int\int_{B_R}|D^\alpha l^n|^2\,dvdx\right)^{1/2}\\
&\leq&Ch\|l^n-\barl\|_{X(l^n)}\|l^{n-1}-\barl\|^2_{X(l^{n-1})},
\end{eqnarray*}
where in the last inequality we used lemmas \ref{Sobolev:1} and \ref{Sobolev:2}. By exactly the same argument,
\[
M_2\leq Ch\|l^n-\barl\|_{X(l^n)}\|l^{n-1}-\barl\|^2_{X(l^{n-1})},
\]
as well.

Consider now 
\begin{eqnarray*}
M_3&\leq& C\|D^\alpha l^n\|_{L^2(B_R\times\mathbb{R}^3)}\|D^2 l^{n-1}\|_{L^6(B_R\times\mathbb{R}^3)}\|Dl^{n-1}(x-hv,v)-Dl^{n-1}(x,v)\|_{L^3(B_R\times\mathbb{R}^3)}.
\end{eqnarray*}
Using Sobolev's inequalities we find that
\begin{eqnarray*}
M_3&\leq& C\|D^\alpha l^n\|_{L^2(B_R\times\mathbb{R}^3)}\|D^3 l^{n-1}\|_{L^2(B_R\times\mathbb{R}^3)}\|Dl^{n-1}(x-hv,v)-Dl^{n-1}(x,v)\|^{1/2}_{L^2(B_R\times\mathbb{R}^3)}\\
&&{}\times{}\|D^2l^{n-1}(x-hv,v)-D^2l^{n-1}(x,v)\|^{1/2}_{L^2(B_R\times\mathbb{R}^3)}.
\end{eqnarray*}
Using lemma \ref{Sobolev:2} this is less than
\[
Ch\|l^n-\barl\|_{X(l^n)}\|l^{n-1}-\barl\|^2_{X(l^{n-1})}.
\]
Thus we showed that
\[
\int I_4\,dx{}\leq{}Ch\|l^n-\barl\|_{X(l^n)}\|l^{n-1}-\barl\|^2_{X(l^{n-1})}.
\]

It remains to estimate $I_3.$ Notice, that it has structure similar to that of $I_4.$ Once we establish the estimates in lemma \ref{lemma:3}, $\int I_3\,dx$ is estimated in a similar way, leading to 
\[
\int I_3\,dx{}\leq{} Ch\|l^n-\barl\|_{X(l^n)}\left( \|l^{n-1}-\barl\|_{X(l^n-1)}{}+{} \|l^{n-1}-\barl\|^2_{X(l^n-1)}{}+{} \|l^{n-1}-\barl\|^3_{X(l^n-1)}\right),
\]
which completes the proof of lemma \ref{lemma:energy:3}.
\end{proof}

Collecting all energy estimate we conclude the next lemma.
\begin{lemma}
\label{lemma:Final}
 Assume (without loss of generality) that for all $n,$ $\|l^n-\barl\|_{X(l^n)}\leq 1.$ There is $C>0$ independent of $(n,h)$ such that for all $n=0..N,$
\begin{equation}
\label{est:Final}
\|l^n-\barl\|^2_{X(l^n)}{}\leq{}\|l^0-\barl\|^2_{X(l^0)}{}+{}CT.
\end{equation}
\end{lemma}
Now we impose a smallness condition on $T:$
\[
CT< \Delta_2-\Delta_1,
\]
(see lemma \ref{lemma:nondegenerate}), which implies that for all $n, $ and $x,$  $l^n(x,\cdot)\in P(R,\delta_1,r,\delta_2),$ verifying Hypothesis 1.

\subsubsection{Convergence of the scheme}
In this section we show that properly interpolated on time axis solution of the discrete scheme converges to a classical solution of \eqref{model:1}.

Let $\{l^n(x,v)\}$ be the sequence verifying estimate \eqref{est:Final}. Define a continuous function $l^h(x,t,v)$ by
\begin{equation}
\label{approx:continuous:1}
l^h(x,t,v){}={}\left\{
\begin{array}{ll}
l^n(x-(t-nh)2v,v), & t\in[nh,nh+h/2],\\\\
\left(1{}-{}2\dfrac{(n+1)h-t}{h}\right)l^{n+1}(x,v){}+{}
2\dfrac{(n+1)h-t}{h}l^n(x-hv,v),& t\in[nh+h/2,(n+1)h].
\end{array}
\right.
\end{equation}
Notice that $\partial_t l^h$ is piecewise constant in $t.$
By construction, for all $(x,v)$ and all $t\not=nh,$ $l^h$ is a solution of the equation
\begin{equation}
\label{eq:approx}
\partial_t l^h{}={}-2v\cdot\grad_xl^h\id_{A_h}(t){}+{}
\frac{2}{h}(l^{\lfloor t/h\rfloor+1}-l^{\lfloor t/h\rfloor})\id_{[0,T]\setminus A_h}(t),
\end{equation}
where
\[
A_h{}={}\cup_{n=0}^{N-1}[nh,nh+h/2].
\]
The following bounds are easily verified given \eqref{est:Final}.
\begin{lemma}[Bounds]
\label{lemma:bounds}
The following statements hold.
\[
l^h-\barl\quad\mbox{\rm bounded in } L^\infty((0,T)\times B_R; H^3(\mathbb{R}^3));
\]

\[
D_{x,t,v}l^h\quad\mbox{\rm bounded in } L^\infty(\mathbb{R}^3\times(0,T)\times B_R);
\]

\[
\forall (x,t),\quad \supp_vl^h(x,t,\cdot)\subset B_R.
\]

\end{lemma}
\begin{lemma}[Compactness]
\label{lemma:compactness} There a sequence (still labeled) $h\to0,$ and a continuous function $l(x,t,v),$ such 
\[
l-\barl\in L^\infty((0,T)\times B_R;H^3(\mathbb{R}^3));
\]
\[
D_{x,t,v}l\in L^\infty(\mathbb{R}^3\times(0,T)\times B_R);
\]
\[
\mbox{ $l^h$ converges to $l$ uniformly on compact sets of $\mathbb{R}^3\times[0,T]\times\mathbb{R}^3;$}
\]
for any $(x,t),$
\[
\supp_vl(x,t,\cdot)\subset B_R,
\]
and
\[
l(x,t,\cdot)\in \E.
\]
\end{lemma}
\begin{proof}
Given the bounds of lemma \ref{lemma:bounds}, it remains to prove the inclusion $l(x,t,\cdot)\in \E.$ Fix $t\in(0,T].$ Let integer $n_h$ be such that
\[
n_hh\leq t<n_hh+h.
\]
Using the definition of function $l^h$ we find that for 
$t\in[n_hh,n_hh+h/2],$
\[
|l^{n_h}(x,v)-l^{h}(x,t,v)|{}\leq{}|l^{n_h}(x,v)-l^{n_h}(x-(t-n_hh)2v,v)|,
\]
and for $t\in[n_hh+h/2,n_hh+h),$
\begin{multline*}
|l^{n_h}(x,v)-l^{h}(x,t,v)|{}\leq{}\left(1-2\frac{(n_h+1)h-t}{h}\right)|l^{n_h+1}(x,v)-l^{n_h}(x,v)|\\{}+{}
2\frac{(n_h+1)h-t}{h}|l^{n_h}(x-hv,v)-l^{n_h}(x,v)|.
\end{multline*}
Using \eqref{est:1}, \eqref{est:2}, and \eqref{est:Final}, we find that
\[
\sup_x\int_{B_R}|l^{n_h}(x,v)-l^h(x,t,v)|^2\,dv\to0.
\]
It follows that for fixed $(x,t),$ $l(x,t,\cdot){}={}\lim l^h(x,t,\cdot)$ is also  a limiting point of a sequence $\{l^{n_h}(x,\cdot)\}\subset\E.$ Since $\E$ is closed, 
$
l(x,t,\cdot)\in \E.
$
\end{proof}

We conclude the analysis by taking the limit in the equation \eqref{eq:approx}. By $W(l)$ we denote the anti-derivative of $w,$ normalized by $W(0){}={}0.$

\begin{lemma}[Limiting equations]
For any $\psi\in C^\infty_0(\mathbb{R}^4),$ -- test function,
and for any $i=1..k$ it holds:
\begin{equation}
\label{eq:weak}
\iiint W(l)l_i(v)\{\partial_t\psi{}+{}v\cdot\grad_x\psi\}\,dvdxdt{}={}0;
\end{equation}
for a.e. $(x,t):$
\begin{equation}
\label{eq:strong}
\int\left\{\partial_tW(l){}+{}v\cdot\grad_x W(l)\right\}l_i(v)\,dv{}={}0;
\end{equation}
for a.e. $(x,t):$
\begin{equation}
\label{eq:strong1}
\int\left\{\partial_t l{}+{}v\cdot\grad_x l\right\}l_i(v)\,w(l)dv{}={}0;
\end{equation}
vector 
\[
U_i(x,t){}={}\int l_i(v)W(l)\,dv,\quad i=1..k,
\] 
is a classical solution of the system of PDEs \eqref{eq:Hyperbolic:2} corresponding to the closure \eqref{model:1}.
\end{lemma}

\begin{proof}
Only the first statement needs a proof. Others follow from that, since $l$ is Lipschitz continuous in $(x,t,v).$ We multiply equation \eqref{eq:approx} by $w(l^h)\psi(x,t)l_i(v),$ and consider the integral
\[
2\iiint \id_{A_h}(t)l_i(v)W(l^h)v\cdot\grad_x\psi\,dvdxdt.
\]
Due to the uniform convergence of $l^h,$ the integral converges to 
\[
\iiint l_i(v)W(l)v\cdot\grad_x\psi\,dvdxdt.
\]
Consider the integral
\begin{multline*}
\frac{2}{h}\int_{[0,T]\setminus A_h}\iint l_i(v)w(l^h)(l^{\lfloor t/h\rfloor+1}(x,v)-l^{\lfloor t/h\rfloor}(x-hv,v))\psi\,dvdxdt \\
{}={}\frac{2}{h}\int_{[0,T]\setminus A_h}\iint l_i(v)w(l^{\lfloor t/h\rfloor})(l^{\lfloor t/h\rfloor+1}(x,v)-l^{\lfloor t/h\rfloor}(x,v))\psi\,dvdxdt\\
{}+{} \frac{2}{h}\int_{[0,T]\setminus A_h}\iint l_i(v)(w(l^h)-w(l^{\lfloor t/h\rfloor}))(l^{\lfloor t/h\rfloor+1}(x,v)-l^{\lfloor t/h\rfloor}(x,v))\psi\,dvdxdt\\
+ 
\frac{2}{h}\int_{[0,T]\setminus A_h}\iint l_i(v)(w(l^h)-w(l^{\lfloor t/h\rfloor}))(l^{\lfloor t/h\rfloor}(x-hv,v)-l^{\lfloor t/h\rfloor}(x,v))\psi\,dvdxdt.
\end{multline*}
The first integral on the right equals zero by the definition of $l^{\lfloor t/h\rfloor+1}.$ The other two are of the order $h,$ due to estimates \eqref{est:1}, \eqref{est:2}, and \eqref{est:Final}.
\end{proof}
Finally, since the classical solutions of the problem in question are unique, we conclude that the original sequence $l^h$ converges to $l$ as $h\to0.$
\end{section}

%%%%%%%%%%%%%%%%%%%%%%%%%%%%%%%%%%%%%%%%%%%%%%%%%%%%%%%%%
%%%%%%%%%%%%%%%%%%%%%%%%%%%%%%%%%%%%%%%%%%%%%%%%%%%%%%%%%
%%%%%%%%%%%%%%%%%%%%%%%%%%%%%%%%%%%%%%%%%%%%%%%%%%%%%%%%%

\begin{section}{BGK-type equation}
In section we consider a problem of projecting a BGK-type model \eqref{eq:Misha} onto a finite dimensional space $\E.$ 
Let $\E_0=\span\{1,v,|v|^2\},$ and $\E{}={}\span\{ l_i(v),\,i=1..k\},$ with $\E_0\subset\E.$  Let $\Pi^0_l$ denote the projection of $l$ onto $\E_0$ in  $L^2$ space with the weight $w=w(l).$
We consider the problem
\begin{equation}
\label{model:2}
\begin{cases}
\partial_t l {}+{}v\cdot\grad_x l{}-{}(\Pi^0_l-l){}\in{} (E^*(l))^\perp, & (x,t,v)\in \mathbb{R}^3\times(0,T)\times\mathbb{R}^3,\\
l(x,t,\cdot)\in E^*, & (x,t)\in\mathbb{R}^4_+, \\
l(x,0,v){}={}l^0(x,v), & (x,v)\in\mathbb{R}^6.
\end{cases}
\end{equation}

Let $\barl\in \E_0,$ and $l^0=l^0(x,\cdot)\in\E,$ be the functions verifying assumptions \eqref{cond:initial} of the previous section. 

Given $h\in(0,1],$ and the values of approximation $l^{n-1}(x,t,v),$ define approximation at the next level, $l^n$ as
\begin{eqnarray}
\hatl^{n-1}(x,v)&=&l^{n-1}(x-hv,v),\notag \\
\label{Projection:2}
\checkl^{n-1}(x,v)&=&\hatl^{n-1}(x,v){}+{}h(\Pi^0_{\hatl^{n-1}}{}-{}\hatl^{n-1})\\
&&=(1-h)\hatl^{n-1}(x,v){}+{}h\Pi^0_{\hatl^{n-1}},\quad \Pi^0_{\hatl^{n-1}}{}={}\mbox{\rm proj}_{E_0(l^{n-1})}\hatl^{n-1},\notag \\
l^n(x,v) &=&\mbox{\rm proj}_{E(l^{n-1})}\checkl^{n-1}. \notag
\end{eqnarray}

We prove
\begin{theorem}
\label{theorem:2}
Let $\barl,l^0$ be as described above and $h\in(0,1].$ There is time $T>0,$ independent of $h,$ such that all functions $l^n,$ $n=0..\lceil T/h\rceil,$ in \eqref{Projection:2} are well defined, and  the family $\{l^h\}$ of interpolations of $l^n$'s, defined in \eqref{inter:1}--\eqref{inter:3}, converges as $h\to0$, uniformly on compact set in $\mathbb{R}^3\times[0,T]\times\mathbb{R}^3$ to a unique classical solution of \eqref{model:2}.
\end{theorem}

\begin{proof} 
We start with the heuristic arguments of the proof. It is based on the orthogonality of $\checkl^{n-1}-l^n$ to $\E,$ and orthogonality of $\hatl^{n-1}-\Pi^0_{\hatl^{n-1}}$ to $\E_0.$ Consider the following orthogonal decompositions (w.r.t. weight $w^{n-1}{}={}w(l^{n-1})$):
\begin{eqnarray}
\label{orth:1}\checkl^{n-1}-\barl &=& (l^n-\barl){}+{}(\checkl^{n-1}-l^n),\\
\label{orth:2}\hatl^{n-1}-\barl &=& (\Pi^0_{\hatl^{n-1}}-\barl){}+{}(\hatl^{n-1}{}-{}\Pi^0_{\hatl^{n-1}}).
\end{eqnarray}
From the first equation we obtain
\[
|l^n{}-{}\barl|^2_{w^{n-1}}{}\leq{}|\checkl^{n-1}-\barl|^2_{w^{n-1}}{}\leq{}(1-h)|\hatl^{n-1}-\barl|^2_{w^{n-1}}{}+{}h|\Pi^0_{\hatl^{n-1}}-\barl|^2_{w^{n-1}}.
\]
Here and below $h\in(0,1).$ From the second equation we get
\[
|\Pi^0_{\hatl^{n-1}}-\barl|^2_{w^{n-1}}\leq |\hatl^{n-1}-\barl|^2_{w^{n-1}},
\]
which leads to 
\[
|l^n{}-{}\barl|^2_{w^{n-1}}{}\leq{}|\hatl^{n-1}-\barl|^2_{w^{n-1}}.
\]
It implies, as explained in theorem \ref{theorem:1},
\[
\int |l^n{}-{}\barl|^2_{w^{n}}\,dx{}\leq{}\int |l^{n-1}-\barl|^2_{w^{n-1}}\,dx{}+{}O(h),
\]
where $O(h)$ measures the change in weights.
Differentiating equations \eqref{orth:1} and \eqref{orth:2} in $x,$
we get:
\begin{eqnarray*}
\label{orth:3}D^\alpha_x(\checkl^{n-1}-\barl) &=& D^\alpha_x(l^n-\barl){}+{}D^\alpha_x(\checkl^{n-1}-l^n),\\
\label{orth:4}D^\alpha_x(\hatl^{n-1}-\barl) &=& D^\alpha_x(\Pi^0_{\hatl^{n-1}}-\barl){}+{}D^\alpha_x(\hatl^{n-1}{}-{}\Pi^0_{\hatl^{n-1}}).
\end{eqnarray*}
It follows that
\[
|D^\alpha_x(l^n-\barl)|^2_{w^{n-1}}{}\leq{}|D^\alpha_x(\checkl^{n-1}-\barl|^2_{w^{n-1}}{}+{}R_n,
\]
and 
\[
|D^\alpha_x(\Pi^0_{\hatl^{n-1}}-\barl)|^2_{w^{n-1}}{}\leq{}|D^\alpha_x(\hatl^{n-1}-\barl)|^2_{w^{n-1}}{}+{}\tilde{R}_n,
\]
where $\int R_n\,dx,\,\int \tilde{R}_n\,dx{}={}O(h).$ By varying the weights we arrive at
\[
\int|D^\alpha_x(l^n-\barl)|^2_{w^n}\,dx{}={}\int|D^\alpha_x(l^{n-1}-\barl|^2_{w^{n-1}}\,dx{}+{}O(h).
\]
By appropriately changing the arguments of lemmas  \ref{lemma:1}--\ref{lemma:Final} we obtain
\begin{lemma}
There are $T,C>0$ such that for all $n=0..\lceil T/h\rceil,$
\[
\|l^n-\barl\|_{X(l^n)}{}\leq C.
\]
For all $x\in\mathbb{R}^3,$ $l^n(x,\cdot)\in P(R,\delta_1,r,\delta_2).$
\end{lemma}
We omit the proof of this lemma.

A continuous process $l^h(x,t,v),$ $t\in[0,T]$ is constructed by a linear interpolation between states $l^n, \hatl^n, \checkl^n,$ and $l^{n+1}$ on interval $[nh,(n+1)h]:$ \\
for $t\in[nh, (n+1/3)h],$ 
\begin{equation}
\label{inter:1}
l^h(x,t,v){}={}l^n(x-(t-nh)3v,v);
\end{equation}
for $t\in[(n+1/3)h,(n+2/3)h],$
\begin{equation}
\label{inter:2}
l^h(x,t,v){}={}\left( 1-3\frac{t-(n+1/3)h}{h}\right)l^n(x-hv,v){}+{}3\frac{t-(n+1/3)h}{h}\checkl^n(x,v);
\end{equation}
for $t\in[(n+2/3)h,(n+1)h],$
\begin{equation}
\label{inter:3}
l^h(x,t,v){}={}\left( 1-3\frac{t-(n+2/3)h}{h}\right)\checkl^n(x,v){}+{}3\frac{t-(n+2/3)h}{h}l^{n+1}(x,v).
\end{equation}
Repeating the convergence arguments of theorem \ref{theorem:1} we find that $l^h$ converges (uniformly on compact sets of $\mathbb{R}^3\times[0,T]\times\mathbb{R}^3$ to a Lipschitz continuous function $l,$ which is a classical solution of \eqref{model:2}.
\end{proof}

\end{section}

\begin{section}{Appendix}
\subsubsection{Lemmas from the theory of Sobolev's spaces}
\begin{lemma} 
\label{Sobolev:1}
Let $f$ be a measurable function such that for some constant $\bar{f},$ $f-\bar{f}\in H^3(\mathbb{R}^3).$ There is $C>0,$ independent of $f,$ and $\bar{f},$ such that
\[
\esssup_x|\grad f|\leq{}C\|f-\bar{f}\|_{H^3(\mathbb{R})^3}.
\]
\end{lemma}

\begin{lemma}
\label{Sobolev:2}
Let $f\in H^1(\mathbb{R}^3).$ Then, for all $h\in\mathbb{R}^3,$
\[
\int |f(x+h)-f(x)|^2\,dx{}\leq{}|h|^2\int|\grad f|^2\,dx.
\]
\end{lemma}

\subsubsection{Equivalence of norms in $v$--variable}
By assumption on the linear independence of polynomials $l_i(v)$ spanning $\E,$ restriction of $\E$ to any ball is a $k+1$--dimensional vector space. Any two norms on $\E$ are equivalent to $\|l\|{}={}(\sum_i \gamma_i^2)^{1/2},$ for $l=\sum_i \gamma_il_i(v).$ In particular, if   $B_R$ and $B_r$ are two balls, norms 
\[
\|l\|_{L^p(B_R)},\,\|l\|_{L^p(B_r)},\quad p\in[1,+\infty],
\]
are equivalent. 
This property implies the following lemma.
\begin{lemma}
Consider set
\[
\bar{E}^*{}={}\left\{ \sum \gamma_i(x)l_i(v)\,:\, \gamma_i\in H^m(\mathbb{R}^3),\,i=1..k\right\},
\]
for some $m\geq 0.$  For any $p,q\in[1,+\infty],$ norms
\[
\|l\|_{L^p(B_R;H^m(\mathbb{R}^3))},\, \|l\|_{L^q(B_R;H^m(\mathbb{R}^3))},
\]
on $\bar{E}^*,$ are equivalent. 

For any $p\in[1,+\infty],$ norms
\[
\|l\|_{L^p(B_R;H^m(\mathbb{R}^3))},\, \|l\|_{L^p(B_r;H^m(\mathbb{R}^3))},
\]
on $\bar{E}^*,$ are equivalent.
\end{lemma}

\end{section}

\end{document}